\documentclass[12pt,reqno]{amsart}
\usepackage{amsmath,amssymb,amsthm,mathtools,calc,verbatim,enumitem,tikz,url,hyperref,mathrsfs,cite,fullpage}
\usepackage{bbm}
\usepackage{stmaryrd}
\usepackage{textcomp}
\usepackage{setspace}
\usepackage{amssymb}
\usepackage{amsthm}
\usepackage{amsmath}
\usepackage{graphicx}
\usepackage{marvosym}
\usepackage{empheq}
\usepackage{latexsym}
\usepackage{fontenc}
\usepackage{color}

\newcommand\N{\mathbb{N}}

\renewcommand\le{\leqslant}
\renewcommand\ge{\geqslant}
\renewcommand\to{\rightarrow}

	\def\N{\mathbb{N}}

	\def\<{\langle }
	\def\>{\rangle }

\usepackage{amssymb}
\newtheorem{theor}{Theorem}[section]
\newtheorem{prop}[theor]{Proposition}
\newtheorem{lemma}[theor]{Lemma}
\newtheorem{coro}[theor]{Corollary}
\newtheorem{conj}[theor]{Conjecture}
\newtheorem{problem}[theor]{Problem}
\theoremstyle{definition}
\newtheorem{defi}[theor]{Definition}
\theoremstyle{remark}
\newtheorem{rema}[theor]{Remark}

\newtheorem*{claim*}{Claim}
\newtheorem*{qu*}{Question}

\newcommand{\genA}{[A]}

\newcommand{\Zd}{\mathbb Z^d}

\newcommand{\pp}{\mathbb P}
\newcommand{\p}{\mathbb P}
\newcommand{\e}{\mathbb E}
\newcommand{\U}{\mathcal U}
\newcommand{\UU}{\mathcal U}
\newcommand{\dhp}{\mathbb H_u}
\newcommand{\Ss}{\mathcal{S}}

\renewcommand{\iff}{\Longleftrightarrow}

\newcommand{\fip}{\lambda_i(p)}
\usepackage{dsfont}
\usepackage{caption}
\usepackage{subcaption}
\usepackage{pifont}
\usepackage{anysize}
\marginsize{1in}{1in}{1in}{1in}
\begin{document}

\title{The $d$-dimensional bootstrap percolation models with threshold at least double exponential}
\author{Daniel Blanquicett}

\address{Mathematics Department,
University of California, Davis, CA 95616, USA}
\email{drbt@math.ucdavis.edu}
\thanks{{\it Date}: January 22, 2022.
}
	
\begin{abstract}
Consider a $p$-random subset $A$ of initially infected vertices in the discrete cube $[L]^d$,
and assume that the neighbourhood of each vertex consists of the $a_i$ nearest neighbours
in the $\pm e_i$-directions for each
$i \in \{1,2,\dots, d\}$, where $a_1\le a_2\le \dots \le  a_d$.
Suppose we infect any healthy vertex $v\in [L]^d$ already having $r$ infected neighbours,
and that infected sites remain infected forever.
In this paper we determine the $(d-1)$-times iterated logarithm of the critical length for percolation up to a constant factor, 
for all $d$-tuples $(a_1,\dots ,a_d)$ and all $r\in \{a_2+\dots + a_d+1,
\dots, a_1+a_2+\dots + a_d\}$.

Moreover, we reduce the problem of determining this (coarse) threshold for all $d\ge 3$ and all $r\in \{a_d+1,
\dots, a_1+a_2+\dots + a_d\}$, to that of determining the threshold for all $d\ge 3$ and all $r\in \{ a_d+1,
\dots, a_{d-1} + a_d\}$.
\end{abstract}
	
\maketitle 
\section{Introduction}
The study of bootstrap processes on graphs was initiated in 1979 by Chalupa,
Leath and Reich~\cite{ChLR79}, and is motivated by problems arising from statistical physics, such as the Glauber dynamics of 
the zero-temperature Ising model, and kinetically constrained spin models of the liquid-glass transition 
(see, e.g.,~\cite{DB21,FSS02,Morris09,MMT18,Morris17}). 
The $r$-neighbour bootstrap process on a locally finite graph $G$ is a monotone cellular automata on the 
configuration space $\{0,1\}^{V(G)}$, (we call vertices in state $1$ ``infected"), evolving in discrete time
in the following way: $0$ becomes $1$ when it has at least $r$ neighbours in state $1$, and infected vertices remain infected forever.
Throughout this paper, $A$ denotes the initially infected set, and we write $\genA=G$
if the state of each vertex is eventually 1.

We will focus on \emph{anisotropic} bootstrap models, which are $d$-dimensional analogues of a family of
(two-dimensional) processes studied by Duminil-Copin, van Enter and Hulshof \cite{EH07,DCE13,DEH18}.
In these models the graph $G$ has
vertex set $[L]^d$, and the neighbourhood of each vertex consists of the $a_i$ nearest neighbours in the
$-e_i$ and $e_i$-directions for each $i \in [d]$,
where $a_1\le \cdots\le a_d$ and $e_i\in\Zd$ denotes the $i$-th canonical unit vector.
In other words, $u,v\in [L]^d$ are neighbours if (see Figure \ref{figanis3d} for $d=3$)
\begin{align}\label{neigh3} 
u-v\in N_{a_1,\dots,a_d}:=\{\pm e_1,\dots, \pm a_1e_1\}\cup \cdots \cup \{\pm e_d,\dots, \pm a_de_d\}.
\end{align}
We also call this process the $\mathcal N_r^{a_1,\dots, a_d}$-{\it model}.
Our initially infected set $A$ 
is chosen according to the Bernoulli product measure $\p_p=\bigotimes_{v\in [L]^d}$Ber$(p)$,
and we are interested in the so-called {\it critical length for percolation},
for small values of $p$
\begin{equation}\label{criticalL}
 L_c(\mathcal N_r^{a_1,\dots,a_d},p):= \min\{L\in\mathbbm N: \pp_p(\genA=[L]^d
 )\ge 1/2\}.
 \end{equation}

 
The analysis of these bootstrap processes for $a_1=\cdots= a_d=1$ was initiated by Aizenman and Lebowitz~\cite{AL88} in 1988,
who determined the magnitude of the critical length 
up to a constant factor in the exponent for the $\mathcal N_2^{1,\dots,1}$-model (in other words, they determined the
`metastability threshold' for percolation). In the case $d = 2$, Holroyd~\cite{H03} determined (asymptotically, as $p \to 0$) the constant in the exponent  (this is usually called a sharp metastability threshold).

For the general $\mathcal N_r^{1,\dots,1}$-model with $2\le r\le d$, the threshold was determined by Cerf and Cirillo \cite{CC99} and Cerf
and Manzo \cite{CM02}, and the sharp threshold by Balogh, Bollob\'as and Morris \cite{BBM09}
and Balogh, Bollob\'as, Duminil-Copin and Morris  \cite{BBDM12}: for all $d\ge r\ge 2$ there exists a computable constant 
$\lambda(d,r)$ such that, as $p\to 0$,
\begin{equation*}
 L_c(\mathcal N_r^{1,\dots,1},p) = \exp_{(r-1)}\bigg(\frac{\lambda(d,r) + o(1)}{p^{1/(d-r+1)}}\bigg).
\end{equation*}


The $\mathcal N_r^{a_1,a_2}$-model is called isotropic when $a_1=a_2$ and anisotropic when $a_1<a_2$.
Hulshof and van Enter \cite{EH07} determined the threshold for the first interesting anisotropic model given by the family $\mathcal N_{3}^{1,2}$, and the 
corresponding sharp threshold was determined by Duminil-Copin and van Enter  \cite{DCE13}.

The threshold was also determined in the general case $r=a_1+a_2$ by van Enter and Fey 
 \cite{AA12} and the proof can be extended to all $a_2+1\le r\le a_1+a_2$: as $p\to 0$,
\begin{equation}\label{paso1}
L_c\left(\mathcal N_{r}^{a_1,a_2},p\right)=\exp\left(
\Theta\left(\lambda_{r-a_2}(p)\right)\right),
\end{equation}
where for each $i\in[a_1]$,

\begin{equation}\label{fip}
\fip=\lambda_i(p,a_1,a_2)=
\begin{cases}
p^{-i} & \textup{if }a_2=a_1,\\
p^{-i}(\log p)^2 & \textup{if } a_2>a_1.
\end{cases}
\end{equation}

 
\subsection{Anisotropic bootstrap percolation on $[L]^d$}
In this paper we consider the $d$-dimensional analogue of the anisotropic bootstrap process studied by Duminil-Copin,
van Enter and Hulshof. 
In dimension $d=3$, we write $a_1=a, a_2=b$ and $a_3=c$.
\vskip -.2cm
\begin{figure}[ht]
	\centering
	\includegraphics[width=0.35\textwidth]{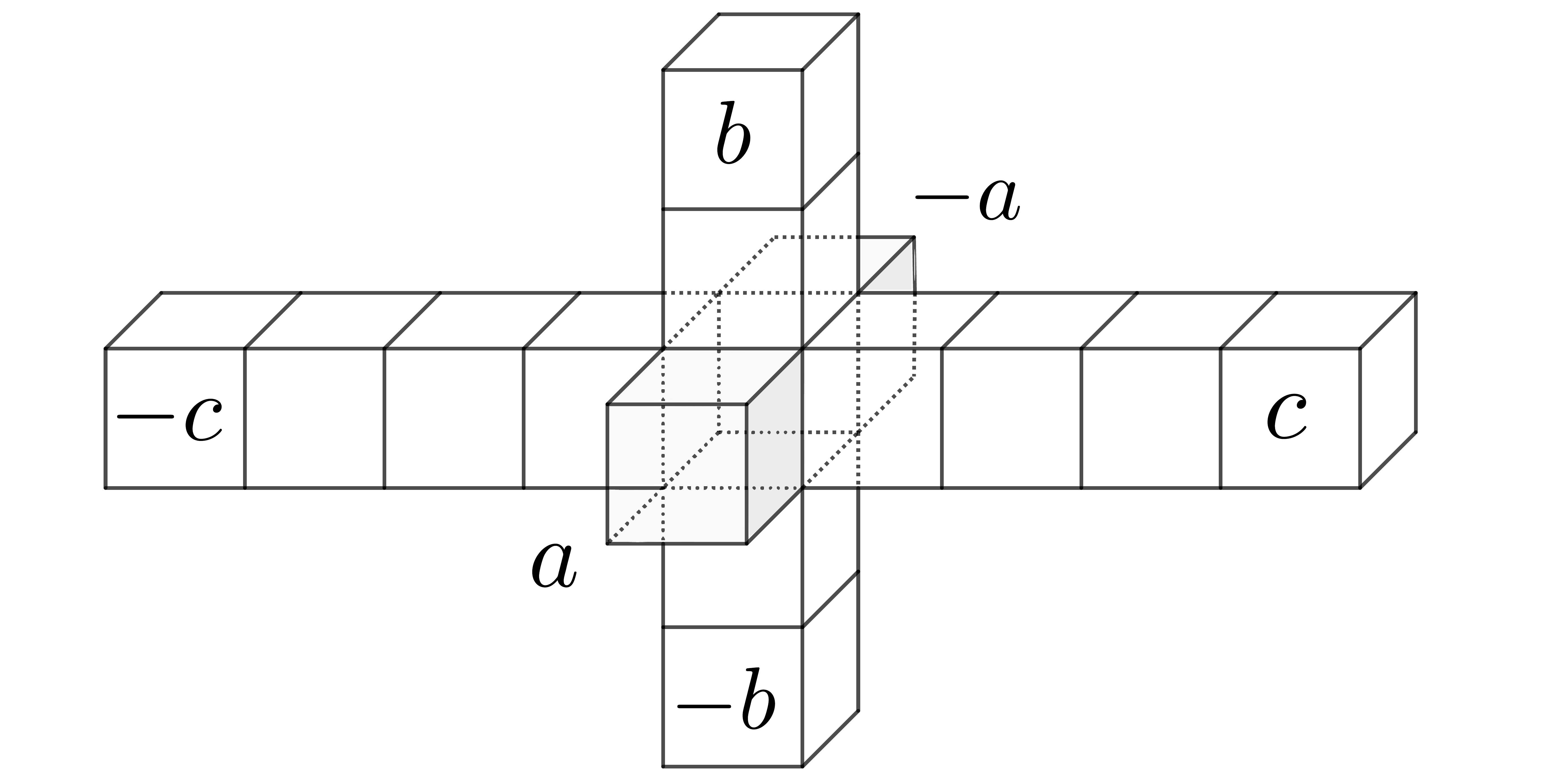}
	\caption{The neighbourhood $N_{a_1,a_2,a_3}$ with $a_1=1, a_2=2$ and $a_3=4$. The $e_1$-axis is towards the reader, the 
	$e_2$-axis is vertical, and the $e_3$-axis is horizontal.}
	\label{figanis3d}
\end{figure}
These models were studied by van Enter and Fey \cite{AA12} for $r=a+b+c$; 
they determined the following bounds on the critical length, 
as $p\to 0$,

\begin{equation}\label{vefey}
		\log \log L_c\left(\mathcal N_{a+b+c}^{a,b,c},p\right)= 
	\Theta\left(\lambda_a(p)\right)
	\end{equation}

Note that, by (\ref{vefey}) the critical length is doubly exponential in $p$ when $r=a+b+c$. 
It is not difficult to show that the critical length is polynomial in $p$ if $r\le c$. 

On the other hand, we have shown in \cite{DB20} that the critical length is singly exponential in the case $r\in\{c+1,\dots,c+b\}$: as $p\to 0$,
\begin{equation}\label{2criti}
\Omega\left(p^{-1/2}\right) \le \log L_c\left(\mathcal N_{r}^{a,b,c},p\right)\le O\left(p^{-b}(\log \tfrac 1p)^2\right).
\end{equation}\\
We moreover determined the magnitude of the critical length up to a constant factor in the exponent in the cases $r\in\{c+1,c+2\}$, for all triples $(a,b,c)$, except for  $r=c+2$ when $c=a+b-1$ (see Section 6 in \cite{DB20}): set $s:=r-c\in\{1,2\}$, then, as $p\to 0$, 
\begin{equation}\label{magnitudenP1}
\log L_c\left(\mathcal N_{r}^{a,b,c},p\right)=  \begin{cases}
\Theta\left(p^{-s/2}\right)  & \textup{if } c=b=a, \\
\Theta\left(p^{-s/2}(\log \frac 1p)^{1/2}\right) & \textup{if } c=b>a, \\
\Theta\left(p^{-s/2}(\log \frac 1p)^{3/2}\right) & \textup{if } c\in \{b+1,\dots,a+b-s\}, \\
\Theta\left(p^{-s}\right)                                  &   \textup{if } c=a+b,   \\
\Theta\left(p^{-s}(\log \frac 1p)^{2}\right)   & \textup{if } c> a+b. 
	\end{cases}
\end{equation}
While we conjecture that
$\log L_c\left(\mathcal N_{a+b+1}^{a,b,a+b-1},p\right)= \Theta\left(p^{-1}(\log \tfrac 1p)^{2}\right).$

In this paper we generalize \eqref{vefey}
by showing that the critical length is doubly exponential in $p$ for each $r\in\{c+b+1,\dots,c+b+a\}$.
Indeed, we determine $\log_{(d-1)}\big(L_c\left(\mathcal N_{a_d+ \cdots+ a_{2}+i}^{a_1,\dots,a_d},p\right)  \big)$  up  to  a  constant  factor, for all dimensions $d\ge 3$ and every $i\in [a_1]$.

The following is our main result.
\begin{theor}\label{doublyexp}
For each $d\ge 3$ and $i\in[a_1]$, as $p\to 0$, 
\begin{equation}
L_c\left(\mathcal N_{a_d+ \cdots+ a_{2}+i}^{a_1,\dots,a_d},p\right)     = 
\exp_{(d-1)} \Theta(\lambda_i(p)).
\end{equation}
\end{theor}

The techniques in this paper can be used to reduce the general problem of determining $L_c\left(\mathcal N_{r}^{a_1,\dots,a_d},p\right)$ (coarse threshold) for all $d\ge 3$ and all $r\in \{a_d+1,
\dots, a_1+a_2+\dots + a_d\}$, to that of determining $L_c\left(\mathcal N_{r}^{a_1,\dots,a_d},p\right)$  for all $d\ge 3$ and all $r\in \{ a_d+1,
\dots, a_{d-1} + a_d\}$
(the 2-critical families only, see Definition  \ref{rcri} and Section \ref{future} below).

\begin{coro}\label{only2cri}
For every $m\in \{2,\dots, d\}$ and $i\in [a_{m-1}]$, as $p\to 0$, the following holds:
if $ L_c\left(\mathcal N_{a_{m}+i}^{a_1,\dots,a_m},p\right)
=\exp \Theta\left(\xi_i(p)
\right)$, for some function $\xi_i(p)= \xi_i (p,a_1,\dots , a_m) $ then
\[L_c\left(\mathcal N_{a_d+ \cdots+ a_{m}+i}^{a_1,\dots,a_d},p\right)     = 
\exp_{(d-m+1)} \Theta(\xi_i(p)). 
\]
\end{coro}

Note that in this corollary, it is an open problem to determine the functions $\xi_i(p)$ for all $m\ge 4$ and $i\in [a_{m-1}]$. While for $m=3$, we only know $\xi_i(p)$ for $i=1,2$ (except for $i=2$ when $a_3= a_1+a_2-1$) by \eqref{magnitudenP1}, and it is unknown for $i\in \{3,\dots, a_{2}\}$.


\subsection{The BSU model}
The model we study here is a special case of the following extremely general class of $d$-dimensional monotone cellular automata, which were introduced by Bollob\'as, Smith and Uzzell~\cite{BSU15}.

Let $\U=\{X_1,\dots,X_m\}$ be an arbitrary finite family of finite
subsets of $\Zd\setminus \{0\}$. We call $\U$ the {\it update family}, 
each $X\in\U$ an {\it update rule}, and the process itself $\U${\it-bootstrap percolation}.
Let $\Lambda$ be either $\Zd$ or $[L]^d$ or $\Zd_L$ (the $d$-dimensional torus of sidelength $L$).
Given a set $A\subset \Lambda$ of initially {\it infected} sites, set $A_0=A$, and define for each $t\ge 0$,
\[A_{t+1}=A_t\cup\{x\in\Lambda: x+X\subset A_t \text{ for some }X\in\U\}.\]
The set of eventually infected sites is the {\it closure} of $A$, denoted by
$\genA_\U=\bigcup_{t\ge 0}A_t$, and
we say that there is {\it percolation} when $\genA_\U=\Lambda$.

For instance, our $\mathcal N_{r}^{a_1,\dots,a_d}$-model is the same as $\mathcal N_{r}^{a_1,\dots,a_d}$-bootstrap percolation,
where $\mathcal N_{r}^{a_1,\dots,a_d}$ is the family consisting of all subsets of size $r$ of the neighbourhood 
$N_{a_1,\dots,a_d}$ in (\ref{neigh3}), and we denote $[A]=[A]_{\mathcal N_{r}^{a_1,\dots,a_d}}$.


 Let $S^{d-1}$ be the unit $(d-1)$-sphere and denote the discrete half space orthogonal to $u\in S^{d-1}$ as
 $\dhp^d:=\{x\in\Zd:\langle x,u\rangle <0\}$.
The {\it stable set} $\Ss=\Ss(\U)$ is the set of all $u\in S^{d-1}$
such that no rule $X\in\U$ is contained in $\dhp^d$. 
Let $\mu$ denote the Lebesgue measure on $S^{d-1}$. The following classification of families was proposed in \cite{BSU15} for $d=2$ and extended to all dimensions in \cite{BDMS15}:
A family $\U$ is

\begin{itemize}
 \item {\it subcritical} if for every hemisphere $\mathcal H \subset S^{d-1}$ we have $\mu(\mathcal H \cap\Ss)>0$.
 \item {\it critical} if there exists a hemisphere $\mathcal H \subset S^{d-1}$ such that $\mu(\mathcal H \cap\Ss)=0$, and
every open hemisphere in $S^{d-1}$ has non-empty intersection with $\Ss$;
 \item {\it supercritical} otherwise. 
 \end{itemize}


For dimension $d=2$, Bollob\'as, Duminil-Copin, Morris and Smith proved a universality result in \cite{BDMS15}, 
determining the critical length (with $A\sim \bigotimes_{v\in \Zd_L}$Ber$(p)$)
\[L_c(\U,p):= \min\{L\in\mathbbm N: \pp_p(\genA_\U=\Zd_L)\ge 1/2\},\]
up to a constant factor in the exponent for all two-dimensional critical families $\U$, which we can briefly state as follows.
\begin{theor}[Universality]
	Let $\U$ be a critical two-dimensional family. There exists a computable positive integer
	$\alpha=\alpha(\U)$ such that, as $p\to 0$, either
	\begin{equation}
	\log L_c(\U,p) =\Theta(p^{-\alpha}),
	\end{equation}
	or
	\begin{equation}
	\log L_c(\U,p) =\Theta(p^{-\alpha}(\log \tfrac 1p)^2).
	\end{equation}
\end{theor}
 
 Proving a universality result of this kind for higher dimensions is a challenging open problem.
 However, there is a weaker conjecture about all critical families and all $d\ge 3$, stated by the
 authors in \cite{BDMS15}.
 
\begin{conj} Let $\UU$ be a critical d-dimensional family. 
There exists $r\in\{2,\dots,d\}$ such that, as $p\to 0$
	\begin{equation}\label{r-cri}
	 \log_{(r-1)} L_c(\UU,p)=p^{-\Theta(1)},
	\end{equation}
\end{conj}
\begin{defi}\label{rcri}
We say that a $d$-dimensional update family $\UU$ is $r$-{\it critical} if it satisfies condition (\ref{r-cri}) (so, roughly speaking, $\UU$ behaves like the classical $r$-neighbour model).

\end{defi}



Observe that the family  $\mathcal N_{r}^{a_1,\dots,a_d}$ is critical if and only if 
 \[r\in\{a_d+1,\dots, a_1+\dots + a_d\}.\]

As an illustration, let us verify this for $d=3$: If $r>a+b+c$ then every $u\in S^2$ is in the stable set, since there is no rule of $\mathcal N_{r}^{a,b,c}$ contained in $\dhp^3$. Thus  $\Ss(\mathcal N_{r}^{a,b,c})=S^2$, and the model is subcritical. 
For each $i=1,2,3$, let us denote by \[S_i^1:=\{(u_1,u_2,u_3)\in S^{2}: u_i=0\}\] the unit circle contained in $S^2$ that is orthogonal to the vector $e_i$.\\
When $r\le c$, for every $u\notin S_3^1$ either $\{r'e_3: r'\in[r]\}$ or $\{r'e_3: -r'\in[r]\}$ is contained in $\dhp^3$, so $u$ is not in the stable set. Therefore
$\Ss(\mathcal N_{r}^{a,b,c})\subset S_3^1$, so the
hemisphere $\mathcal H_3$ pointing in the $e_3$-direction satisfies $\mathcal H_3 \cap\Ss=\varnothing$ and
$\mathcal N_{r}^{a,b,c}$ is supercritical.
\begin{figure}[ht]
	\centering
	\includegraphics[width=.9\textwidth]{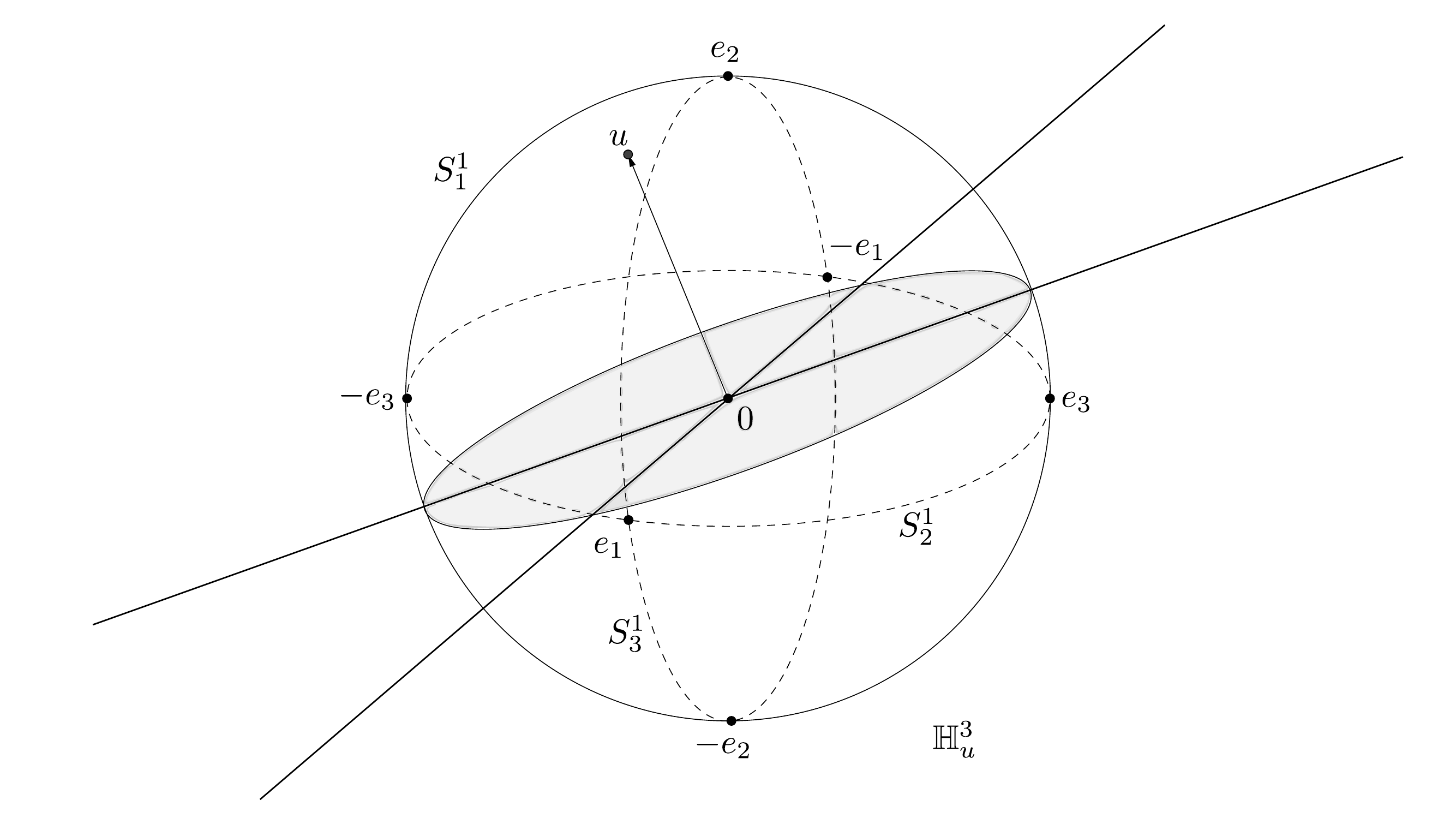}
	\caption{$S_1^1$ is the big circle, $S_2^1$ and $S_3^1$ are drawn with dashed 
	ellipses. The vector $u$ is outside $S_1^1\cup S_2^1\cup S_3^1$ and $\dhp^3$ contains all positive
	multiples of $e_1,-e_2$ and $e_3$.}
	\label{fig3dsphere}
\end{figure}\\
Finally, when $r\in\{c+1,\dots,a+b+c\}$, every canonical unit vector is in the stable set since $r>c\ge b\ge a$,
so every open hemisphere in $S^{2}$ intersects $\Ss(\mathcal N_{r}^{a,b,c})$.
Moreover, for each $u\notin S_1^1\cup S_2^1\cup S_3^1$, $\dhp^3$ intersects all three coordinate axis (see Figure \ref{fig3dsphere}),
hence there is a rule contained in $\dhp^3$ since $r\le a+b+c$.
It follows that
 $\Ss(\mathcal N_{r}^{a,b,c})\subset S_1^1\cup S_2^1\cup S_3^1$
and every hemisphere $\mathcal H \subset S^{2}$ satisfies $\mu(\mathcal H \cap\Ss)=0$, so $\mathcal N_{r}^{a,b,c}$ is critical, as claimed.

Indeed, a careful analysis would lead us to all possibilities for the stable set of the family $\mathcal N_{r}^{a_1,\dots, a_d}$ in dimensions $d\ge 3$. Some cases are:
\begin{equation}\label{viaSs}
\Ss(\mathcal N_{r}^{a_1,\dots, a_d})=
\begin{cases}
\{\pm e_1, \dots, \pm e_d\} & \textup{ for } a_d< r \le a_1+a_2,\\
S^1_{1,2} \cup \{\pm e_3,\dots, \pm e_d\} & \textup{ for } a_1+a_2< r \le a_1+a_3,\\
S^1_{1,2} \cup S^1_{1,3} \cup \{\pm e_4,\dots, \pm e_d\}  & \textup{ for } a_1+a_3< r \le a_2+a_3,\\
\hspace{.7cm} \vdots &  \\
  S^{d-2}_1 \cup S^{d-2}_2 \cup \cdots  \cup S^{d-2}_d & \textup{ for } a_2+\dots +a_d< r \le a_1+a_2+\dots +a_d,
\end{cases}
\end{equation}
where, $S^1_{i,k}$ is the unit circle contained in $S^{d-1}$ that contains vectors $e_i, e_k$, while $S^{d-2}_i \subset S^{d-1}$ is the $(d-2)$-sphere orthogonal to vector $e_i$.

For instance, if $d=3$ Note that by \eqref{2criti},  the family $\mathcal N_{r}^{a,b,c}$ is 2-critical for all $r\in\{c+1,\dots,c+b\}$ (first 3 cases in \eqref{viaSs}).
On the other hand, Theorem \ref{doublyexp} implies that $\mathcal N_{r}^{a,b,c}$ is 3-critical for all $r\in \{c+b+1,\dots,c+b+a\}$ (last case in \eqref{viaSs}).


 

\section{Upper bounds}\label{SectionUpper12}
To prove upper bounds, it is enough to give one possible way of growing from $A$
step by step until we fill the whole of $[L]^d$. 

\begin{defi}\label{intfilled}
A {\it rectangular block} is a set of the form $R=[l_1]\times \cdots \times[l_d]\subset\Zd$. 
We say that a rectangular block $R$ is {\it internally filled} if $R\subset [ A\cap R]$,
and denote this event by $I^\bullet(R)$.
\end{defi}

Given $d\ge 2$ and $ a_2\le \dots \le a_d$, let us denote \[
s_d:= a_2+a_3+\dots + a_d.\] 

As usual in bootstrap percolation, we actually prove a stronger proposition.

\begin{prop}\label{upper2}
Given $d\ge 3$, fix $i\in [a_1]$ and  consider $\mathcal N_{s_d+i}^{a_1,\dots,a_d}$-bootstrap percolation.
There exists a constant $\Gamma=\Gamma(d, a_d)>0$ such that, if \[L=\exp_{(d-1)}\big(
\Gamma \lambda_i(p)\big),\] then
 $\p_p\left(I^\bullet([L]^d)\right)\to 1,\ as\ p\to 0.$
\end{prop} 
One key step in the proof of this proposition is to refine the upper bounds in (\ref{paso1}) for all dimensions, which can be done by using standard renormalization techniques.
\begin{lemma}[Renormalization]\label{elactual}
Given $d\ge 2$, fix $i\in [a_1]$ and  consider $\mathcal N_{s_d+i}^{a_1,\dots,a_d}$-bootstrap percolation.
There exists a constant
$N_0 = N_0(d, a_d)>0$ such that,
	\begin{equation}\label{1menosL2}
	\pp_p\left([ A]
	= [N]^d\right)\ge
	1-\exp\left(-\Omega(N)\right),
	\end{equation}
	for all $p$ small enough and 
	$N\ge N_0$.
\end{lemma}
\begin{proof}
 For $d=2$, it follows from \eqref{paso1} and renormalization techniques (see e.g. \cite{Sch92}). For $d\ge 3$ it follows by induction on $d\ge 3$, meaning, Proposition \ref{upper2} with $d$ implies Lemma \ref{elactual} with $d$, while Lemma \ref{elactual} with $d-1$ implies Proposition \ref{upper2} with $d$ (see the proof of Proposition \ref{upper2} below).
\end{proof}

Now, we are ready to show the upper bound for 
$L_c\big(\mathcal N_{s_d+i}^{a_1,\dots,a_d},p\big)$.
\begin{proof}[Proof of Proposition \ref{upper2}]
We use induction on $d\ge 3$. Assume that the proposition holds for all dimensions $2,3,\dots, d-1$.	Set $L=\exp_{(d-1)}\big(
\Gamma \lambda_i(p)\big)$, where $\Gamma>0$ is a large constant to be chosen.
Let $C$ be another large constant ($\Gamma$ will depend on $C$), $N=\exp_{(d-2)}(C\lambda_i(p))$, and
	consider the rectangular block
	\[R:=[N]^{d-1}\times [a_d]\subset[L]^d,\]
 and the events  $F_L:=\{\exists$ a copy of $R$ contained in $A\}$,
and $G_L:=\{[ A\cup R]=[L]^d\}$.
Note that $\p_p\left(I^\bullet([L]^d)\right)\ge \p_p(F_L)\p_p(G_L|R\subset A),$ so we need to show that $\p_p(F_L)\to 1$
and $\p_p(G_L|R\subset A)\to 1$, as $p\to 0$. 
 
Indeed, there are roughly $L^d/|R|$ disjoint (therefore independent) copies of $R$ 
	(which we label $Q_1,\dots, Q_{L^d/|R|}$),
	and $|R|	\le \exp_{(d-2)}\big({p^{-2i}}\big)$, so
\begin{align*}
	\p_p(F_L^c)& \le \p_p\left(\bigcap_i   
	(Q_i\not\subset A)\right) 
 \le \left[1-\p_p(R\subset A)\right]^{L^d/|R|} 
 \le \exp\big(-p^{|R|}L^d/|R|\big)\\
 & \le \exp\left(-\exp\left(de^{\Gamma \lambda_i(p)}-c e^{2C\lambda_i(p)}\log \tfrac{1}{p}-p^{-2i}\right)\right)
 \le \exp\left(-\exp\left(de^{\Gamma \lambda_i(p)}-e^{3C\lambda_i(p)}\right)\right).
	\end{align*}
By taking $\Gamma\ge 3C$ we conclude $\p_p(F_L)\to 1$, as $p\to 0$.

Next, set $M=\exp_{(d-2)}\big(p^{-2a_2}\big)$, and
	consider the rectangular block
	\[R':=[N]^{d-1}\times [M]\supset R.\]
In order to prove that $\p_p(G_L|R\subset A)\to 1$, as $p\to 0$ it is enough to verify that
\begin{equation}\label{RtoR'}
    \p_p(I^\bullet(R')|R\subset A)\to 1, \textup{ as } p\to 0,
    \end{equation}
then $R'$ will grow with high probability to fill the whole of $[L]^d$, since each of its $(d-1)$-faces is of supercritical size for the corresponding induced $(d-1)$-dimensional bootstrap process on that face.
More precisely, on the face orthogonal to the (easiest grow) $e_d$-direction with volume $N^{d-1}\ge \exp_{(d-2)}{(2C\lambda_i(p))}$, by induction hypothesis the corresponding critical length is $L_c\left(\mathcal N_{s_{d-1}+i}^{a_1,\dots, a_{d-1}},p\right)= \exp_{(d-2)}{\Theta(\lambda_i(p))}\le N^{d-1}$ if $C$ is large; on the face orthogonal to the (second hardest) $e_2$-direction with volume $MN^{d-2}\ge e^{p^{-2a_2}}$ (and shape such that it is much larger than a critical droplet in all $d-1$ directions) the corresponding critical length is $L_c\left(\mathcal N_{s_d-a_2+i}^{a_1,a_3,\dots, a_{d}},p\right)= \exp_{(d-2)}{O(\lambda_i(p))}\le MN^{d-2}$, and on the face orthogonal to the (hardest) $e_1$-direction with volume $MN^{d-2}$ as well the corresponding critical length is $L_c\left(\mathcal N_{s_d-a_1+i}^{a_2,\dots, a_{d}},p\right)= \exp_{(d-2)}{O(\lambda_{a_2-a_1+i}(p))}\le MN^{d-2}$.

Finally, by Lemma \ref{elactual} (applied with $d-1$), 
\[\p_p(I^\bullet(R')|R\subset A)\ge \left(1-e^{-\Omega(N)} \right)^M \ge 
\exp\left(-2Me^{-\Omega(N)}\right) \to 1,\]
as $p\to 0$, and \eqref{RtoR'} follows.
\end{proof}
	

\section{Lower bounds}\label{SectionLowerComp}  

In this section we will prove the lower bounds, and the proof is an application of the {\it Cerf-Cirillo method} (see Section \ref{CCmethod}) and the {\it components process} (see Definition \ref{compro} below),
a variant of an algorithm introduced Bollob\'as, Duminil-Copin, Morris, and Smith \cite{BDMS15}. We will prove the following. 
\begin{prop}\label{lower1weaker}
Given $d\ge 3$, fix $i\in [a_1]$ and  consider $\mathcal N_{s_d+i}^{a_1,\dots,a_d}$-bootstrap percolation.
There exists a constant $\gamma=\gamma(d, a_d)>0$ such that, if \[L\le \exp_{(d-1)}\big(
\gamma \lambda_i(p)\big),\] then
 $\p_p\left(I^\bullet([L]^d)\right)\to 0,\ as\ p\to 0.$
\end{prop}

In order to show this proposition, we need to introduce a notion about rectangular blocks
which is an approximation to being internally filled, and
this notion requires a strong concept of connectedness; we define both concepts in the following.
\begin{defi}\label{strconn}
For $d\ge 1$, let $G^d=(V,E)$ be the graph with vertex set $[L]^d$ and edge set given by
$E=\{uv: \|u-v\|_\infty \le 2a_d\}$.   
 We say that a set $S\subset [L]^d$ is {\it $d$-strongly connected}
 if it is connected in the graph $G^d$.
\end{defi}

\begin{defi}\label{intspa}
We say that the rectangular block $R\subset [L]^d$ is 
{\it internally spanned} by $A$, if there exists a strongly connected set $S\subset [ A\cap R]$
such that $R$ is the smallest rectangular block containing $S$. We denote this event by $I^{\times}(R)$.
\end{defi}
Note that when a rectangular block is internally filled then it is also internally spanned.
Now, given $T\subset [L]^d$, let us denote by long$(T)$ the largest sidelength of the smallest rectangle containing $T$, and let
\[\textup{diam}(T):=\max\{\textup{long}(S): S\subset T, S \textup{ strongly connected}\}.\]
Since $I^{\bullet}([L]^d)$ is an increasing event,
Proposition \ref{lower1weaker} is a consequence of the following result.
\begin{prop}\label{lower1}
Given $d\ge 3$, fix $i\in [a_1]$ and  consider $\mathcal N_{s_d+i}^{a_1,\dots,a_d}$-bootstrap percolation.
There exists a constant $\gamma=\gamma(d, a_d)>0$ such that, if \[L=\exp_{(d-1)}\big(
\gamma \lambda_i(p)\big),\] then, as $p\to 0.$
 
\[\p\big(\textup{diam}([A])\ge \log L\big) \le L^{-1}.\]
\end{prop}
The rest of this paper is devoted to the proof of this result.
\subsection{The components process}
The following is an adaptation of the spanning algorithm in \cite[Section 6.2]{BDMS15}.
We will use it to show
an Aizenman-Lebowitz-type lemma, 
which says that when a rectangular block
is internally spanned, then it contains internally spanned rectangular blocks of all
intermediate sizes 
(see Lemmas \ref{ALlema0} and \ref{ALlema1} below).
\begin{defi}[The components $d$-process]\label{compro}
Consider $\mathcal N_r^{a_1,\dots, a_d}$-bootstrap percolation on $[L]^d$ with $r>a_d$.
 Let $A=\{v_1,\dots,v_{|A|}\}\subset [L]^d$.
 Set $\mathcal R:=\{S_1,\dots,S_{|A|}\}$, where $S_i=\{v_i\}$ for each $i=1,\dots,|A|$.
 Then repeat the following steps until STOP:
 \begin{enumerate}
  \item If there exist distinct sets $S_{1},S_{2}\in\mathcal R$ such that
  \[ S_{1}\cup S_{2} \]
  is strongly connected, then remove them from $\mathcal R$,
  and replace by $[ S_{1}\cup S_{2} ]$.
  \item If there do not exist such sets in $\mathcal R$, then STOP.
 \end{enumerate}
\end{defi}
\begin{rema}\label{stopfinitetime}
We highlight that the condition $r>a_d$ (equivalent to $\mathcal N_r^{a_1,\dots, a_d}$ is not supercritical) guarantees that at any stage of the component process, if 
$S=[ S_{1}\cup S_{2} ]$ is added to the collection
$\mathcal R$, then the smallest rectangular block (which is finite) containing $S$ is internally spanned.
\end{rema}
 Since $G^d$ is finite, the process stops in finite time; so that we can consider the final
 collection $\mathcal R'$ and set
 $V(\mathcal R')=\bigcup\limits_{S\in\mathcal R'}S$.

\begin{lemma}\label{final=gen}
  $V(\mathcal R')=\genA$.
\end{lemma}
\begin{proof}
 See Lemma 3.10 in \cite{DB20}.
\end{proof}

The following is a variant of the Aizenman-Lebowitz Lemma in \cite{AL88}.
\begin{lemma}\label{ALlema0}
 Consider $\mathcal N_{r}^{a_1,\dots,a_d}$-bootstrap percolation with $r\ge a_d+1$.
For every  $k\le \textup{diam}([A])$, there exists an internally spanned rectangular block
 $R\subset [L]^d$ satisfying 
 \[k\le \textup{diam}(R)\le 2a_dk.\]
\end{lemma}
\begin{proof}
Let $S$ be the first set that appears in the components process such that diam$(S)\ge k$, and let $R$ be
the smallest block containing $S$. Since diam$(S)=$ diam$(R)$, it only remains to show that 
diam$(S)$ at most $2a_dk$.
In fact, we know that $S=[ S_{1}\cup S_{2} ]$ for some 
sets $S_{t}$ such that, diam$(S_t)\le k-1$ for each $t=1,2$.
Since $S$ is strongly connected, we conclude
\begin{align*}\label{maxdiam2}
\textup{diam}(S)\le \textup{diam}(S_1)+\textup{diam}(S_2)+2a_d\le
 2a_dk.
\end{align*}
\end{proof}

Basically, the same proof of this lemma (by using the components ($d-1$)-process) allows us to conclude the following.
\begin{lemma}\label{ALlema1}
Consider $\mathcal N_{r}^{a_1,\dots,a_{d-1}}$-bootstrap percolation with $r\ge a_{d-1}+1$.
For every  $k,l\le \textup{diam}([A])$, there exists an internally spanned copy of the rectangular block
 $W\times [h]$, with $W\subset [L]^{d-2}$, satisfying $\textup{diam}(W)\le 2a_{d-1}l$, $h\le 2a_{d-1} k$ and either \[\textup{diam}(W)\ge l\textup{ or } h\ge k.\]
\end{lemma}

\subsection{Anisotropic bootstrap percolation on $[N]^{d-1}$ with subcritical sizes}
Let us fix $d\ge 3$, $i\in [a_1]$ and  consider $\mathcal N_{s_{d-1}+i}^{a_1,\dots,a_{d-1}}$-bootstrap percolation on $[N]^{d-1}$, where 
\[N\le  \exp_{(d-2)}\left(\gamma \lambda_{i}(p)\right)
 \]
and  $\gamma=\gamma(d,a_{d-1})>0$ is a small constant  (so that percolation is unlikely). 
Note that for $d=3$, $\exp\left(\gamma \lambda_{i}(p)\right)
 \approx L_c(\mathcal N_{a_2+i}^{a_1,a_{2}},p)^\gamma$ by \eqref{paso1}, while we will deduce that $\exp_{(d-2)}\left(\gamma \lambda_{i}(p)\right)
 \approx L_c(\mathcal N_{s_{d-1}+i}^{a_1,\dots,a_{d-1}},p)^\gamma$ by induction on $d$.

\begin{defi}
We define the {\it component} (or {\it cluster})
at $(\lfloor N/2\rfloor,\dots, \lfloor N/2\rfloor)\in [N]^{d-1}$ as the $(d-1)$-strongly connected component containing $(\lfloor N/2\rfloor,\dots,  \lfloor N/2\rfloor)$ in the graph induced by 
$[A\cap [N]^{d-1}]$, and we denote it by $\mathcal K=\mathcal K(A,i,a_1,\dots,a_{d-1})\subset [N]^{d-1}$.
\end{defi}
The following results are standard in bootstrap percolation.

\begin{prop}\label{condition(c)}
Consider $\mathcal N_{s_{d-1}+i}^{a_1,\dots a_{d-1}}$-bootstrap percolation. For any $\varepsilon>0$, there exists $\gamma = \gamma(d,a_{d-1})>0$ such that if $N\le  \exp_{(d-2)}\left(\gamma \lambda_{i}(p)\right)$, as $p\to 0$,
\begin{enumerate}
\item[\textup{(a)}]
$\p_p(\textup{diam}(\mathcal K)\ge p^{-i-\varepsilon})\le N^{-\varepsilon}$, when $d=3$.

    \item[\textup{(b)}] 
$\p_p\big(\textup{diam}(\mathcal K)\ge \exp_{(d-3)}(\lambda_i(p)) \big)\le N^{-\varepsilon}$, when $d\ge 4.$
\end{enumerate}
\end{prop}
The proof of this proposition goes by induction on $d$ (like the proof of Lemma \ref{elactual}), by combining it with Proposition \ref{lower1}.
The base case is given by (a), thus, this is the only case that we will prove.
Moreover, when $d=3$, the proof in the isotropic case $a_1=a_2$ follows from usual application of the the Aizenman-Lebowitz Lemma (see for instance, the paragraph after Theorem 7.1 of \cite{BDMS15} with $\alpha =i$). While the proof in the case $a_1<a_2$ is basically the same as that of Theorem 8.1 of \cite{BDMS15}, with some minor modifications; for completeness, we will  prove this case only.
\begin{proof}[Proof of Proposition \ref{condition(c)}\textup{(a)}]
Assume that $a_1<a_2$ and let $\delta=\delta(\varepsilon)>0$ be small. If $\textup{diam}(\mathcal K)\ge  p^{-i-\varepsilon}$, then by Lemma \ref{ALlema1}, there exists a rectangle $R=[w]\times [h]$ such that $w\le p^{-i-\varepsilon}$, $h\le \delta p^{-i}\log\frac 1p$ and either, $w\ge\Omega( p^{-i-\varepsilon})$ or $h\ge \Omega(\delta p^{-i}\log\frac 1p)$.\\
If $w\ge\Omega( p^{-i-\varepsilon})$, since $R$ is internally spanned, every copy of the slab $[2a_2^2]\times [h]$ must contain $i$ vertices of $A$ within constant distance, so for $\delta$ small,
\[\p_p(I^{\times}(R))\le (1-e^{-\Omega(p^{i}\cdot \delta p^{-i}\log\frac 1p)})^{\Omega( p^{-i-\varepsilon})}
\le \exp(-p^{C\delta}p^{-i-\varepsilon})\le 
\exp(- p^{-i-\varepsilon/2}).\]
Analogously, if $h\ge \Omega(\delta p^{-i}\log\frac 1p)$,  every copy of the slab $[w]\times [2a_2^2]$ must contain $a_2+i-a_1$ vertices of $A$ within constant distance, so
\[\p_p(I^{\times}(R))\le (1-e^{-\Omega(p^{a_2+i-a_1}\cdot  p^{-i-\varepsilon})})^{ \Omega(\delta p^{-i}\log\frac 1p)}
\le O(p^{a_2-a_1-\varepsilon})^{ \Omega(\delta p^{-i}\log\frac 1p)}
\le 
\exp(-\delta^2 p^{-i}(\log p)^2).\]

Since there are at most $N^3$ copies of the rectangle $R$ in $[N]^2$, then 
\[\p_p(\textup{diam}(\mathcal K)\ge p^{-i-\varepsilon})\le 
N^3\exp(-\delta^3 p^{-i}(\log p)^2)
\le\exp\left(3\gamma
f_{i}(p) 
-\delta^3 p^{-i}(\log p)^2\right)
\le 
N^{-\varepsilon},\]
for $\gamma\ll  \delta^3$.
\end{proof}

\begin{prop}\label{condition(d)}
Consider $\mathcal N_{s_{d-1}+i}^{a_1,\dots a_{d-1}}$-bootstrap percolation. As $p\to 0$,
\begin{itemize}
    \item[\textup{(a)}] $\e_p(|\mathcal K|)\le \sqrt{p},$  given that $\textup{diam}(\mathcal K)\le p^{-i-\varepsilon}$, when $d=3$.
    \item[\textup{(b)}] $\e_p(|\mathcal K|)\le o(1)$,  given that $\textup{diam}(\mathcal K)\le \exp_{(d-3)}(\lambda_i(p)) $, when $d\ge 4$.
\end{itemize}
\end{prop}
Again, the proof is by induction on $d$, as that of Proposition \ref{condition(c)}. The base case is (a) and we will focus on that again, whose a straightforward application of Aizenman-Lebowitz Lemma (for the isotropic case, see for instance, (3.30) in \cite{CC99}). We will prove Proposition \ref{condition(d)}(a) in the anisotropic case, and the proof is similar to that of Lemma 5.4 in \cite{AA12} (with $i=a$ only), which does not seem to be complete.
\begin{proof}[Proof of Proposition \ref{condition(d)}\textup{(a)}]
It is enough to consider two cases. If $1\le \textup{diam}(\mathcal K)\le 6a_2^2$ then there is a vertex in $A$ within constant distance of ``the origin'' $(\lfloor N/2\rfloor, \lfloor N/2\rfloor)$. On the other hand, if $\textup{diam}(\mathcal K)> 6a_2^2$, by Lemma \ref{ALlema0}, there exists an internally spanned rectangular block
 $R=[w]\times [h]\subset [N]^2$ with $3a_2\le \textup{diam}(R)\le 6a_2^2.$\\
 In particular, $w,h\le 6a_2^2$ and either $w> 3a_2$ or $h> 3a_2$. So we have two subcases:\\ If $w> 3a_2$, then 
 \[\p_p(I^{\times}(R))\le (1-e^{-\Omega(p^{i}\cdot 6a_2^2)})^{3a_2}
\le O(p^{i\cdot 3a_2}).\]
And, for $h> 3a_2$, 
\[\p_p(I^{\times}(R))\le (1-e^{-\Omega(p^{a_2+i-a_1}\cdot  6a_2^2)})^{3a_2}
\le O(p^{3a_2(a_2-a_1+i)})
\le O(p^{3a_2i}).
\]
Finally, there are at most $O(N^2)$ possible choices for the rectangular block $R$, thus
\[\e_p(|\mathcal K|)\le O(6a_2^2p) + O(N^2\cdot N^2\cdot  p^{3a_2i})
\le O(p)+ O(p^{-4i-4\varepsilon +3a_2i})\le p^{1/2},\]
for $\varepsilon>0$ small, since $3a_2\ge 6$.
\end{proof}

\subsection{The proof via Cerf-Cirillo method}\label{CCmethod}
In this section we reproduce a result that was proved in \cite{BBM09} (and used again in \cite{BBDM12}), which is an adaptation of some ideas from \cite{CC99, CM02} and \cite{H06}. Then, we use it to prove Proposition \ref{lower1}.

Let us consider two-colored graphs, that is, simple graphs with two types of edges, which we will label “good” and “bad”.

\begin{defi}
We say that a two-colored graph is {\it admissible} if it either contains at least one bad edge, or if
every component is a clique (i.e., a complete
graph). 
\end{defi}

For any set $S$, we let

\[\Lambda(S):=\{\text{admissible two-colored graphs with vertex set }S\times [2]\}.\]

And, for each $m\in\mathbb N$ we let
\[\Omega(S,m):=\{\mathcal P=(G_1,\dots,G_m): G_t\in \Lambda(S)\text{ for each }t\in [m]\},\]
be the set of sequences of two-colored admissible graphs on $S\times [2]$ of length $m$.
We shall sometimes think of $G_t$ as a two-colored graph on $S\times [2t-1,2t]$, and trust that this will cause no confusion.

Now, for each $\mathcal P\in\Omega(S,m)$, let $G_{\mathcal P}$ denote the graph with vertex set $V(G_{\mathcal P})=S\times [2m]$, and the following edge set $E(G_{\mathcal P})$:
\begin{enumerate}
    \item[(i)] $G_{\mathcal P}[S\times \{2y-1,2y\}]= G_y$,
    \item[(ii)]$\{(x,2y), (x',2y+1)\}\in E(G_{\mathcal P}) \iff x=x'$,
    \item[(iii)] $\{(x,y), (x',y')\}\notin E(G_{\mathcal P})$ if $|y-y'|\ge 2$. 
\end{enumerate}

Edges in $G_{\mathcal P}$ of type (i) are labelled good and bad in the obvious way, to match the label of the corresponding edge in $G_y$. Thus $G_{\mathcal P}$ has three types of edge: good, bad, and unlabelled.

Given $G\in\Lambda(S)$, let $E^g(G)$ denote the set of good edges, and $E^b(G)$ denote the bad edges, so that $E(G)=E^g(G)\cup E^b(G)$. If $uv\in E^g(G)$, then we shall write $u\sim v$.
For each vertex $v=(x,y)\in V(G_{\mathcal P})$, let
\[\Gamma_{\mathcal P}(v):=\{u\in V(G_{\mathcal P}): u \sim v\text{ and }u\ne v\},\]
and let $d_{\mathcal P}(v)=|\Gamma_{\mathcal P}(v)|$. Note that $d_{\mathcal P}(v)$ is the number of good edges incident
with $v$.

Finally, let $X(\mathcal P)$ denote the event that there is a connected path across $G_{\mathcal P}$ (i.e., a path from the set $S\times \{1\}$ to the set $S\times \{2m\}$. The following lemma was first stated in \cite{BBM09}, then in \cite{BBDM12}, but the proof is due to Cerf and Cirillo \cite{CC99}.

\begin{lemma}[Cerf and Cirillo \cite{CC99}, see Lemma 35 of \cite{BBM09}]\label{mainLema}
For each $0<\alpha <1/2$ and $\varepsilon >0$,
there exists $\delta >0$ such that the following holds for all $m\in\N$ and all finite sets $S$ with $\alpha^4|S|^{\varepsilon}\ge 1$.

Let $\mathcal P=(G_1,\dots,G_m)$ be a random sequence of admissible two-coloured graphs on $S\times [2]$, chosen according to some probability distribution $f_{\Omega}$ on $\Omega(S,m)$. Suppose $f_{\Omega}$ satisfies the following conditions:

\begin{enumerate}   
    \item[\textup{($a$)}] Independence: $G_i$ and $G_j$ are independent if $i\ne j$.
    \item[\textup{($b$)}] BK condition: For each $t\in [m]$, $r\in\N$, and each $x_1,y_1,\dots,x_r,y_r\in V(G_t)$,
    \[\p\left(\bigcap_{j=1}^r(x_j\sim y_j) \cap \bigcap_{j\ne j'}^r(x_j\not\sim x_j')\cap \big(E^b(G_t)=\emptyset\big)\right) \le \prod_{j=1}^r\p(x_j\sim y_j),\]
    \end{enumerate}
    and for each $t\in [m]$ and $v\in V(G_{\mathcal P})$,
\begin{enumerate}
    \item[\textup{($c$)}] Bad edge condition: $\p\big(E^b(G_t)\ne \emptyset \big) \le |S|^{-\varepsilon}$,
    \item[\textup{($d$)}] Good edge condition: $\e(d_{\mathcal P}(v))\le \delta$.
\end{enumerate}
Then
\[\p\big(X(\mathcal P)\big)\le \alpha^m|S|.\]
\end{lemma}


We are ready to prove the lower bound.
\begin{proof}[Proof of Proposition \ref{lower1}]
We use induction on $d\ge 3$. Assume that the proposition holds for all dimensions $2,3,\dots, d-1$. In particular, Propositions \ref{condition(c)} and \ref{condition(d)} hold for dimension $d$.
Fix $i\in [a_1]$ and  consider $\mathcal N_{s_d+i}^{a_1,\dots,a_d}$-bootstrap percolation. Fix a small constant $\varepsilon>0$ and let $\gamma >0$ be the constant given by Proposition \ref{condition(c)}, then take $L=\exp_{(d-1)}\big(
\gamma \lambda_{i}(p)\big)$. Let us show that $\p\big(\textup{diam}([A])\ge \log L\big) \le L^{-1}$, as $p\to 0$. 

Suppose that $\textup{diam}([A])\ge \log L$, then by Lemma \ref{ALlema0} there exists an internally spanned rectangular block
 $R\subset [L]^d$ satisfying 
 \[\frac{\log L-1}{2a_d}\le \textup{diam}(R)\le \log L-1.\]
Let $N=\textup{diam}(R)$, then we can assume for simplicity that $R\subset [N]^d$. Moreover, there is a strongly connected path $X$ 
in $[A\cap R]$ joining two opposite $(d-1)$-faces of $[N]^d$, and we can assume that this happens along the (easiest) the $e_d$-direction, so $X$ goes from the set $\{(x_1,\dots , x_d)\in [N]^d: x_d=1\}$ to the set $\{(x_1,\dots , x_d)\in [N]^d: x_d=N\}$. 

Now, let $m=\lfloor N/4a_d \rfloor$ and partition $[N]^d$ into blocks $B_1,\dots, B_{2m}$, each of size $[N]^{d-1}\times [2a_d]$ (for simplicity, assume that $N$ is a multiple of $4a_d$). So, $B_j=\{(x_1,\dots , x_d)\in [N]^d: x_d\in [2a_d(j-1)+1, 2a_d j]\}$, for each $j\in [2m]$.

For each $j$, let us consider a $(d-1)$-dimensional bootstrap process on $B_j':=[N]^{d-1}\times \{j\}$ as follows: Take the initially infected set 
$A'\sim \bigotimes_{v\in [N]^2\times [2m]}$Ber$(2a_d p)$ and then run the $\mathcal N_{s_{d-1}+i}^{a_1,\dots ,a_{d-1}}$-bootstrap process, independently on each $B_j'$.
Note that this defines a concatenated process on $[N]^{d-1}\times [2m]$ consisting of $2m$ independent $(d-1)$-dimensional processes, and couples our original $\mathcal N_{s_{d}+i}^{a_1,\dots ,a_{d}}$-process in the following way:\\
The probability of having a vertex in $A\cap \big([N]^{d-1}\times [(j-1)2a_d, j2a_d]\big) \subset B_j$ is at most $2a_dp$, which is the initial density (for $A'$) in $B_j'$. Also, each vertex in $B_j$ has at most $a_d$ neighbors in $[N]^d\setminus B_j.$ Thus, the projection of components of $[A\cap B_j]$ onto the $(d-1)$-plane orthogonal to $e_d$ is coupled by the components in $[A'\cap B_j']$.
In particular, the existence of $X$ implies the existence of a strongly connected path $X'\subset \bigcup_j[A'\cap B_j']$  from the set $\{(x_1,\dots , x_d)\in [N]^d: x_d=1\}$ to the set $\{(x_1,\dots , x_d)\in [N]^d: x_d=2m\}$. 

Next,  set $S=[N]^{d-1}$, and for each $j\in [2m]$, let $[A](j):=[A'\cap B_j']$, and define a two-colored graph $G_j$ on $S\times [2]$ by
\[uv\in E(G_j) \iff u',v'\textup{ are in the same strong component of }[A](j),\]
where $u'$ is the element of $[N]^{d-1}\times \{2j-1,2j\}$ corresponding to $u$ in the natural isomorphism, and define ``good'' edges by
\[u\sim v \iff \textup{there exists an internally filled strongly connected component}\]
\[\ \ \ \ \ \ \ \ \ \ \ X\subset [A](j) \textup{ such that } u,v\in X \textup{ and diam}(X)\le (\log N)^{1+\varepsilon}.\]

Note that $G_j$ is admissible (see for instance, the proofs of the lower bounds for Theorem 1 in \cite{BBDM12} and \cite{BBM09}). 
Therefore, it is enough to check that the sequence $\mathcal P=(G_1,\dots, G_m)\in \Omega(S,m)$ satisfies the conditions of Lemma \ref{mainLema}.

In fact, condition (a) follows by construction, while condition $(b)$ follows from the van Berg-Kesten Lemma (again, see the proof of Theorem 1 in \cite{BBM09}).

Now, since
$N\le \log L = \exp_{(d-2)}(\gamma \lambda_i(p))$, by  Proposition \ref{condition(c)} we conclude, as $p\to 0$,
\[\p_p(\textup{diam}([A](j))> (\log N)^{1+\varepsilon}) \le N^{-\varepsilon},\]
and by Proposition \ref{condition(d)}, for $p$ small,
\[\e_p(d_{\mathcal P}(v))\le O(\e_p(|\mathcal K|))\le O(\sqrt{p}) =o(1).\]
Finally, by Lemma \ref{mainLema} we conclude that for $\alpha>0$ small,
\[\p_p(X(\mathcal P)) \le \alpha^{\lfloor N/4a_d\rfloor}N^{d-1} \le 1/L^{2d},\]
then, summing over all possible choices of $R\subset [L]^d$ we get
\[\p_p(\textup{diam}([A])\ge \log L)\le 1/L,\]
and we are done.
\end{proof}
Note that in the above proof for $d\ge 4$, when defining ``good'' edges $u \sim v$, we could replace the size $(\log N)^{1+\varepsilon}$ by $O(\log N)$. That refinement would improve the lower bound for the $(d-1)$-times iterated logarithm of the threshold by just a constant factor.




\section{Future work}\label{future}
In dimension $d=3$, a problem which remains open is the determination of the threshold for  $a_3+3 \le r\le a_2+a_3$ (the 2-critical families).
We believe that the techniques used in \cite{DB20} can be adapted to cover these cases 
(though significant technical obstacles remain); recall that in this case, by \eqref{2criti}, the critical length is singly exponential.

For dimensions $d\ge 4$, by using the techniques in \cite{CM02}, it can be shown that as $p\to 0$,
$\log L_c\left(\mathcal N_{a_d+1}^{a_d,\dots,a_d},p\right)\ge \Omega\left(p^{-1/(d-1)}\right)$, so that $\log L_c\left(\mathcal N_{r}^{a_1,\dots,a_d},p\right)\ge \Omega\left(p^{-1/(d-1)}\right)$ for $r\ge a_d+1$.

On the other hand, as it was shown in the appendix of \cite{DB20}, by using Lemma \ref{elactual} and decomposing $[L]^d$ as $L^{d-2}$ disjoint copies of $[L]^2$ all of them parallel to the $e_{d-1}$ and $e_d$-directions, we can see that for $r\in\{a_d+1,\dots,a_d+ a_{d-1}\}$,
\[\log L_c\left(\mathcal N_{r}^{a_1,\dots,a_d},p\right)\le O\left(\log L_c(\mathcal N_{r}^{a_{d-1},a_d},p)\right)
=O\left(p^{-(r-a_d)}(\log p)^2\right).\]

So, it follows that the critical length is singly exponential in the cases
\[r\in\{a_d+1,\dots,a_d+ a_{d-1}\},\]
and the family is 2-critical, by Definition \ref{rcri}.
 It is an interesting open problem to find the critical length for all critical anisotropic models in all dimensions, and by Corollary \ref{only2cri}, we need to do it for the 2-critical families only.
 
 \begin{problem}\label{elproblema}
 Determine the critical length $ L_c(\mathcal N_r^{a_1,\dots,a_d},p)$ for all $d\ge 3, \ a_1\le \cdots \le a_d$ and all $r\in\{a_d+1,\dots,a_d+ a_{d-1}\}$.
 \end{problem}
 



\section*{Acknowledgements}
The author is very grateful to Janko Gravner and Rob Morris for their stimulating conversations on this project, and their many invaluable suggestions. 




\bibliographystyle{plain}
\bibliography{References}
\end{document}